\documentclass{amsart}
\usepackage{./stile2017}
\usepackage{autonum}		

\title{On gaps between sums of four fourth powers}
\author{Luca Ghidelli}
\date{\today}

\address{150 Louis-Pasteur Private, Office 608, Department of Mathematics and Statistics, University of Ottawa, Ottawa ON K1N 9A7, Canada}
\email{{luca.ghidelli@uottawa.ca}}
\subjclass[2010]{Primary 11P05, 11P55; Secondary 11B05}
\renewcommand{\c}{\tfrac 1 2}
\newcommand{\cc}{\tfrac 1 {16}}
\newcommand{\cb}{\tfrac 1 {8}}
\newcommand{\ch}{32}

\newcommand{\gcdaq}{a\in(\Z/q\Z)^\ast}

\renewcommand{\a}{\alpha}
\renewcommand{\b}{\beta}

\newcommand{\lls}{\ll_\epsilon P^\epsilon}
\newcommand{\norm}[1]{{\left\Vert #1 \right \Vert}}

\newcommand{\mm}{\mfk m}
\newcommand{\MM}{\mfk M}
\newcommand{\NN}{\mfk N}
\newcommand{\RR}{\bar R}
\newcommand{\ER}{E_R}
\newcommand{\ES}{E_S}

\newcommand{\Em}{E_\mm}
\newcommand{\EM}{E_\MM}
\newcommand{\EN}{E_\NN}
\newcommand{\TM}[1]{T_\MM ^{(#1)}}
\newcommand{\Tm}[1]{T_\mm ^{(#1)}}
\newcommand{\TN}[1]{T_\NN ^{(#1)}}
\newcommand{\T}[1]{T^{(#1)}}
\renewcommand{\S}[1]{S^{(#1)}}

\newcommand{\Mi}[1]{\MM^{(#1)}}
\newcommand{\Ni}[1]{\NN^{(#1)}}
\newcommand{\mi}[1]{\mm^{(#1)}}

\renewcommand{\AA}[1]{\mfk B^{(#1)}}
\newcommand{\BB}{\mfk B}

\newcommand{\TT}{{\R/\Z}}
\renewcommand{\PP}{\mbf P}

\renewcommand{\P}{\mbf P}

\begin{document}
	
\begin{abstract}
	We prove that for almost all $N$ there is a sum of four fourth powers in the interval $(N-N^\gamma,N]$, for all $\gamma>4059/16384=0.24774..$.
\end{abstract}
\maketitle

\tableofcontents

\section{Introduction}\label{sec:intro}

For every $n\in \N$ there is some natural number $x<n^{1/4}$ such that $n-x^4 = O(x^3) = O(n^{3/4})$. 
If we repeat this procedure we find that for all $n\in\N$ there exist $x_1, x_2,x_3,x_4\in\N$ such that $x_1^4+\dots +x_4^4 = n+O(n^{\gamma})$ with $\gamma = (3/4)^4\approx 0.3164$. 
In this paper we show that the exponent $\gamma$ can be reduced if we require the above statement to hold only for \emph{almost all $n\in\N$}. 
This is motivated by a forthcoming article of the author \cite{ghidelli:theta}, in which we study arithmetic properties of special values of ``cubic'' and ``biquadratic'' theta series. 
In fact, the arguments of that paper require that 
almost all intervals of the form $(n-n^\gamma,n]$, for some $\gamma<0.25$, contain a sum of four fourth powers. 
Using the circle method, Daniel \cite{daniel} studied a similar problem in regard to sums of three cubes. 
Following his approach we are able to prove the following statement.  
\begin{theorem}\label{cor:main} 
	Define $\gamma_0:=4059/16384\approx 0.24774$ and let $\gamma>\gamma_0$. 
	Then for almost all $n\in\N$ (in the sense of natural density) there is a sum of four fourth powers in the interval $(n-n^\gamma,n]$.
\end{theorem}

To put this theorem in perspective, we now survey the relevant literature on sums of four fourth powers and sums of three cubes. 
First, we know from a paper of Davenport \cite{davenport:biquadrates} that there are $\gg N^{\alpha_4}$ distinct sums of four fourth powers up to $N$, for $\alpha_4:=331/412\approx 0.803398$: this means that the average gap between sums of fourth powers is at most of order $\ll N^{1-\alpha_4}\approx N^{0.197}$. 
However, Davenport's result does not measure how uniformly the sums of four fourth powers distribute on the number line, so  it does not imply that almost all gaps have at most this size. 
In fact some probabilistic models \cite{deshouillers:random,erdos:random} suggest that the sums of four fourth powers, and more generally sums of $k$ perfect $k$-th powers for $k\geq 3$, should have positive natural density. 
In particular the gaps between these numbers are conjectured to have bounded average size. 
However, previous work of the author \cite{ghidelli:gaps} shows that there do exist arbitrarily large gaps between numbers that can be written as sums of four fourth powers. In fact we also showed that a positive proportion of the intervals $(n-\psi(n),n]$ does not contains sums of fourth powers, if $\psi(n)$ grows to infinity sufficiently slowly. 
If we trust the probabilistic models, 
we should in fact expect this last statement to hold for $\psi(n)\asymp \log n / \log \log n$. 

The situation for sums of three cubes is similar, and has been considered more extensively in the literature. 
A ``greedy argument'' as the one in the opening of this introduction shows that for all $n\in\N$ there exist $x_1,x_2,x_3\in\N$ such that $x_1^3+x_2^3+x_3^3 = n+ O(n^{\gamma})$, with $\gamma = 8/27\approx 0.296$. 
The aforementioned paper of Daniel \cite{daniel} proves instead that almost all gaps between sums of three cubes up to $N$ have length $O(N^\gamma)$, for all $\gamma>17/108\approx 0.1574$. 
For the number of sums of three cubes up to $N$, the current record is due to Wooley \cite{wooley:cubes}, who proves that there are $\gg N^{\alpha_3}$ of them, with $\alpha_3\approx 0.916862$; this means that on average the gaps between them have order $\ll N^{1-\alpha_3}\approx N^{0.083}$.  
As we wrote above, it is expected on the basis of probabilistic models that the sums of three cubes have positive density in the set of natural numbers. 
This expectation is further discussed in \cite{hooley:sometopics} and is supported by numerical results \cite{deshouillers:cubes}.  
It is also known that there are $\gg N^{1-\epsilon}$ sums of three cubes up to $N$, for every $\epsilon>0$, conditionally on analytic conjectures involving certain $L$-functions \cite{heathbrown:cubes,hooley:cubes:L1,hooley:cubes:L2}. 
However, if the sums of three cubes have positive natural density, they do not lie uniformly on the number line. 
In fact, as we prove in \cite{ghidelli:gaps}, there exists a constant $\kappa>0$ so that, for $\psi(n):=\kappa\sqrt{\log n}(\log \log n)^{-2}$, a positive proportion of the intervals $(n-\psi(n),n]$ does not contain sums of three cubes. 
More generally, our result belongs to the vast literature on Waring's problem, that is the study of those numbers that can be written as sums of perfect powers. 
The interested reader is referred to the survey of Vaughan and Wooley \cite{waring:survey}. 

%
%
We now provide some details on the basic ideas of this paper. 
A classical approach known as ``diminishing ranges'' due to Hardy and Littlewood \cite{HL:diminishing}, consists in counting those sums $x_1^4+\dots+x_4^4$ in an interval $(n-Y,n]$ whose summands have a prescribed size $x_j^4\asymp P_j^4$. 
More precisely, we fix $\PP:=(P_1,P_2,P_3,P_4,Y)\in\R_+^5$ with 
	\begin{align}\label{P:range}
		P_j^{3/4}&\leq P_{j+1}\leq P_j   & (1\leq j\leq 3)
	\end{align} 
 	and let $R(n)=R(n,\PP)$ denote  the number of solutions to the equation 
 	\begin{equation}\label{eq:sysR}
 	n = x_1^4 + x_2^4 + x_3^4 + x_4 ^4 + y
 	\end{equation}
 	subject to 
 	\begin{equation}\label{eq:subR}
 	0\less y\leq Y,\quad \c P_i\less x_i \leq P_i\quad (1\leq i \leq 4). 
 	\end{equation}
	If $n\asymp P_1^4$, say $n\in(N/2,N]$ with $N=P_1^4$, then we expect that, at least on average, $R(n)\asymp Y P_1^{-3} P_2 P_3 P_4$, because there are $\asymp N$ choices for the parameter $n$ and $\asymp YP_1P_2P_3P_4$ choices for the values of the variables of \cref{eq:sysR}. 
	In fact, using the circle method \cite{vaughan:book} of Hardy and Littlewood we prove the following analog of the main lemma in \cite{daniel}.
\begin{theorem}\label{final:main}
	Let $\gamma_0$ be as in \cref{cor:main} and let $\gamma_1:=4992/16384\approx 0.3046$. 
	Given $N>0$ and $\gamma_0<\gamma\leq \gamma_1$, we let $Y:= N^\gamma$, $P=P_1:= \sqrt[4] N$  and $P_{j+1}=P_j^{13/16}$ for $1\leq j\leq 3$. 
	Then for each $\epsilon>0$ we have 
	\begin{align}\label{eq:final:main}
		\sum_{\c N<n\leq N} \abs{R(n)-\RR(n)}^2 &\ll_\epsilon Y N^{1-\gamma_0+\epsilon},
	\end{align}
	where the implied constant depends only on $\epsilon$, and $\RR(n) := \tfrac 1 {32} Y P_2 P_3 P_4 n^{-3/4}$.
\end{theorem}

From this quantitative result one may deduce nontrivial moment estimates for the the size of gaps between sums of four fourth powers, as in \cite[Corollary 2]{daniel} or \cite[Theorem 1.2]{BWshort2}. Moreover, as we will show in the next section, \cref{final:main} implies  \cref{cor:main}. We also claim more generally that, with essentially the same strategy and some more work, one may possibly show that in almost every interval of the form $(N-N^\gamma,N]$ there is a number $m=x_1^k+\dots+x_h^k$ that can be written as the sum of $h\geq 2$ perfect $k$-th powers, provided that $k\geq 3$ and $\gamma>\gamma_0(h,k)$, where
\begin{equation}\label{eq:daniel:generalized}
\gamma_0(h,k):=1-\frac 1 k (1 + \theta_k + \theta_k^2 + \dots + \theta_k^{h-1}), 
\end{equation}
with 
\begin{equation}
\theta_k := 1 - \frac 1 k + \frac {1}{k 2 ^{k-2}}.
\end{equation}

We notice that $\gamma_0(4,4)=4059/16384$ is the exponent that appears in \cref{cor:main} and that $\theta_4=13/16$ is the exponent we use for diminishing the ranges in \cref{final:main}. 
Therefore our result solves the case $h=k=4$ while Daniel \cite{daniel} deals with the case $h=k=3$. 
Recently, a paper of Br\"udern and Wooley \cite{BWshort2} has settled the case $h=2$ for all $k\geq 3$. 
Even though the treatment of only two variables simplifies part of the argument (e.g. the final induction on the number of variables becomes trivial), the case treated by Br\"udern and Wooley should be considered as the hardest one. 
In fact their paper introduces some technical modifications to the original strategy of Daniel, which are unnecessary here.

In addition to the results that we have just mentioned, a few more remarks are in order with respect to the general claim enunciated above. 
The first is that stronger statements are known to be true if $h$ is somewhat larger than $k$. For example, we know that all natural numbers can be written as a sum of $h$ $k$-th powers, if $h$ is large enough \cite{waring:survey}.  
Secondly the claim is nontrivial in general: in comparison the greedy argument produces the exponent $\gamma(h,k)=(1-1/k)^h$, which is the same as \cref{eq:daniel:generalized}, with $\theta_k$ replaced by the smaller $\theta'_k:=1-1/k$. 
Finally, the recent progress on the Vinogradov mean value theorem \cite{Vino:proof,Vino:expo,Vino:Woo} should make it possible to replace $\theta_k$ with a larger value, if $k$ is large enough; see the note in the introduction of \cite{BWshort2} for a more precise remark on this matter.


In closing, let us briefly illustrate the main ingredients in the proof of \cref{final:main}. 
%
%
First the number $R(n)$ is rewritten, by Fourier analysis, as an integral of an exponential sum. 
Then Bessel's inequality is used to produce an integral formula that estimates from above the left-hand side of \eqref{eq:final:main}. 
A characterstic feature of Daniel's approach is that this part of the proof (\cref{sec:R,sec:S}) is performed in conjunction with a triple application of the circle method,%
\footnote{Corresponding to the three pairs of integrals $R\sim U$, $S\sim V$ and $T\sim W$ introduced in the proof.} %
where only one major arc centered around the origin is considered. 
The upper bound that results from this preliminary phase is then finally estimated using a more classical application of the circle method and an induction on the number of variables of the underlying diophantine equations, to produce the expression in the right-hand side of \eqref{eq:final:main}. 
Technically, the minor arcs are treated with a version \cite[Lemma 1]{Vaughan:Weil} of the Weyl differencing inequality \cite[Lemma 2.4]{vaughan:book}, while the major arcs are treated with classical estimates mostly due to Vaughan \cite[Chapter 4]{vaughan:book}. 
%
%
%
In conclusion, we express our contentment in noticing the fortuitous happenstance: that this approach produces an exponent $\gamma_0\approx 0.24774$, that is just barely good enough for our original purpose. 

 
\subsection*{Acknowledgements}
  
I would like to thank my supervisor Damien Roy for his steady encouragement, his careful reading of this manuscript and for his many comments and suggestions. 
This work was supported in part by 
 a full International Scholarship from the Faculty of Graduate and Postdoctoral Studies of the University of Ottawa and by NSERC. 



\section{Heuristics and quantitative results}
\label{sec:K}

In this section we comment on the statement of \cref{final:main} and its consequences regarding the size of gaps between sums of four fourth powers. 

\subsection{Choice of parameters and notation}

In the remainder of the article we write $N=P^4$ and $Y=P^{4\gamma}$, where
\begin{equation}\label{eq:gamma}
	\gamma\in (\tfrac {4059}{16384}, \tfrac{4992}{16384}]
\end{equation}
 and $P$ is some parameter that we let grow to infinity. 
 We also let
\begin{align}\label{eq:P}
	P_1 &= P^{\tfrac{4096}{4096}}
	&
	P_2 &= P^{\tfrac{3328}{4096}}
	&P_3 &= P^{\tfrac{2704}{4096}}
	&P_4 &= P^{\tfrac{2197}{4096}} 
\end{align}
as in \cref{final:main} so that $P_{j+1}=P_j^{13/16}$ for $j=1,2,3$. 
The inequality $\gamma>\tfrac {4059}{16384}$ implies that 
\begin{equation}\label{hyp:gamma0}
	N= o(YP_1P_2P_3P_4)
\end{equation}
which is crucial in the approach of this paper. 
The hypothesis $\gamma\leq \tfrac {4992}{16384}$ is imposed only for technical reasons, as it ensures that
\begin{equation}\label{hyp:gamma1}
	Y^{-2} \geq {P_2}^{-3}.
\end{equation}
In fact the validity of this inequality simplifies some proofs, e.g. that of \cref{S0T0}. 
We denote $\PP=(P_1,P_2,P_3,P_4,Y)$ and define $R(n)=R(n,\PP)$ accordingly, see \cref{sec:intro}.  
Throughout the paper we make various estimates in terms of the parameter $P$, but we also write the results, when possible, in a way that makes explicit the dependence on the choice of $P_1,\dots, P_4$. 
As usual, the notation $A\ll B$ means that  $\abs A\leq c B$ for some absolute $c>0$. 
The contributions of terms that are logarithmic in $P$ or anyway asymptotically smaller than any positive power of $P$ will systematically be collected into a ``$P^\epsilon$ term''. 
We will write $A\lls B$ to mean that $\abs A \leq c P^\epsilon B$, for every $\epsilon>0$ and for some $c=c(\epsilon)>0$ depending only on $\epsilon$. 

\subsection{The heuristic expected value of $R(n)$}

The diminished ranges \eqref{eq:subR} for the variables of \eqref{eq:sysR} reduce the number of sums of fourth powers at our disposal, and so enlarge the gaps between them. 
However the advantage is that those particular sums of powers are more easily controlled, so that it is possible to estimate $R(n)$ as in \cref{final:main}. 
The expected average value of $R(n)$, given by the formula 
\begin{equation}\label{def:RR}
	\RR(n):= \tfrac 1 {32} Y P_2 P_3 P_4 n^{-3/4}
\end{equation}
 is heuristically obtained as follows. 
 Suppose that $P$ is large and that $n\asymp P^4$ is restricted to an interval $n\in(n_0,n_1]$ with $\Delta n:=n_1-n_0 = o(P_1^4)$ and $Y\leq P_2^4 = o(\Delta n)$. 
Then every solution to \cref{eq:sysR}, constrained by \eqref{eq:subR}, also satisfies 
\begin{equation}\label{eq:subR:bis}
n_1^{1/4}\geq x_1> (n_0-4 P_2^4)^{1/4} =: n_1^{1/4}-\Delta x
\end{equation}
with $\Delta x\approx \tfrac 1 4 \Delta n\cdot n^{-3/4}$. 
There are $\Delta n$ choices for the parameter $n\in(n_0,n_1]$ and $\approx 2^{-3}\Delta x P_2 P_3 P_4 Y$ choices of $x_j$ and $y$ constrained by \eqref{eq:subR} and \eqref{eq:subR:bis}, hence we expect that $R(n)\approx \RR(n)$ with $\RR(n)$ as above. 
We notice en passant that $N^{-3/4}P_2P_3P_4 = N^{-\gamma_0}$, where $\gamma_0=4059/16384$, so
\begin{equation}\label{asymp:RR}
	\RR(n)\asymp Y N^{-\gamma_0}.
\end{equation}
 
Therefore we also heuristically expect that a typical $n\in (N/2, N]$ satisfies $R(n)\geq 1$, as soon as $Y$ is somewhat larger than $N^{\gamma_0}$. 

\subsection{Bounding the number of large gaps}

We now show how to prove from \cref{final:main} that the gaps of size $N^\gamma$ with $\gamma>\gamma_0:=4059/16384$ are rare.
For every $\gamma>0$ we denote by $K'(N,N^\gamma)$ the number of $n\in(N/2, N]$ 
with the property that no element of the interval $(n-N^\gamma,n]$ is a sum of four fourth powers.	
\begin{theorem}\label{thm:main}
	Let $\gamma_0$ and $\gamma_1$ be as in \cref{final:main}. Then
	\begin{equation}\label{eq:thm:main}
	K'(N,N^\gamma) \ll N^{1-\xi}  
	\end{equation}
	for every $\gamma>\gamma_0$ and all $\xi<\min\{\gamma_1-\gamma_0,\gamma-\gamma_0\}$. 
\end{theorem}

\begin{proof}
	 If $\gamma\leq \gamma_1$ we may apply 
	 \cref{final:main}. Let $K''(N,\PP)$ denote the number of $n\in (N/2, N]$ for which $R(n)=R(n,\PP)=0$. 
	 For each of those $n$ we have $\abs{R(n)-\RR(n)}=\RR(n)\geq \RR(N)$, hence
\begin{equation}\label{eq:proof:K}
\sum_{\c N<n\leq N} \abs{R(n)-\RR(n)}^2 \gg K''(N,\PP)\cdot \RR(N)^2.
\end{equation}

 It is clear that $K'(N,N^\gamma)\leq K''(n,\PP)$ because whenever the interval $(n-Y,n]$ is empty of sums of four fourth powers, where $Y=N^\gamma$, then $R(n)=0$. 
 By \cref{eq:proof:K,asymp:RR,eq:final:main} we get
 \begin{equation}\label{eq:inequality:K}
 	K'(N,N^\gamma) \ll \RR(N)^{-2}\sum_{\c N<n\leq N} \abs{R(n)-\RR(n)}^2  \ll Y^{-1} N^{1+\gamma_0+\epsilon}
 \end{equation} 
 for every $\epsilon>0$. 
 This gives \cref{eq:thm:main} if $\gamma_0<\gamma\leq\gamma_1$. 
 If $\gamma>\gamma_1\approx0.3046$ then we simply use the inequality $K'(N,N^\gamma)\leq K'(N,N^{\gamma_1})$. 
\end{proof}

We remark that for $\gamma>(3/4)^4\approx 0.3164$ one in fact has $K'(N,N^\gamma)=0$ if $N\gg 1$, by the greedy algorithm mentioned in the introduction.  
%
%
We now show that \cref{cor:main} 
is follows from \cref{thm:main}. 
\begin{proof}[Proof of \protect\cref{cor:main}]
	Fix $\gamma>\gamma_0$ and let $K_\gamma(N)$ count the natural numbers $n\leq N$ such that no element of the interval $(n-n^\gamma,n]$ is a sum of four fourth powers. 
	Take some $\gamma'\in(\gamma_0,\gamma)$ and let $N_0$ be such that $N^{\gamma'}\leq (N/2)^\gamma$ for all $N\geq N_0$. 
	Then for every real number $N\geq N_0$ we have 
	\begin{equation}\label{K:K'}
		K_\gamma(N) \leq N_0+ \sum_{k=0}^{\lfloor \log_2 N/N_0 \rfloor} K'(N/2^k, (N/2^k)^{\gamma'}).
	\end{equation}
	Then by \eqref{eq:thm:main} we get
	\begin{equation}
		K_\gamma(N) \ll N^{1-\xi} \sum_{k=0}^\infty (2^{-(1-\xi)})^k \ll_{\xi} N^{1-\xi},
	\end{equation}
	where $\xi$ is any positive number with $\xi+\gamma_0<\min\{\gamma_1,\gamma'\}$. 
	In particular, we have that $K_\gamma(N)=o(N)$ as $N\to\infty$. 
\end{proof} 

\section{On the expected value of $R(n)$}
\label{sec:R}

In this section we rewrite the number $R(n)$ in a way that makes it amenable to be studied with analytic methods. Then we give a first estimate of the deviation $R(n)-\RR(n)$ via a partial application of the circle method, with only one major arc centered at zero.

\subsection{Integral representation and Weyl sums}

We denote by $e(\xi):=e^{2\pi i \xi}$ the normalized complex exponential function, considered as an additive character of $\TT$. 
By the ``orthogonality property'' we mean the well-known fact that for all $m\in\Z$ we have
\begin{equation}
	\int_\TT e(m\a) d\a = \begin{cases} 1 & \text{ if } m=0\\ 0 & \text{ if } m\neq 0 .\end{cases}
\end{equation}
By orthogonality we can rewrite $R(n)=R(n,\PP)$ as follows
\begin{align}
	R(n)&:=\int_\TT \sum_{\substack{y,x_1,\ldots,x_4\\ \c P_j<x_j\leq P_j\\ 0\leq y<Y}} e((x_1^4+x_2^4+x_3^4+x_4^4+y- n)\a) d\a
	\\
	& = \int_\TT f_1 f_2 f_3 f_4 g  e(-n\a)d\a ,
	\label{def:R:total} 
\end{align}
where $f_i=f(\a,P_i)$, $g= g(\a,Y)$ are given by the following Weyl exponential sums
\begin{align}
	f(\a,X)&:= \sum_{\c X< x\leq X} e(\a x^4) \\
	g(\a,Y)&:= \sum_{0 \leq  y< Y} e(\a y). 
\end{align}

We observe that $g(\a,Y)$ is the sum of a geometric progression, therefore we have 
\begin{equation}
g(\a,Y) = \frac{ e(\a (Y+O(1)) ) - 1 }{e(\a) - 1}.
\end{equation}

From this formula, we easily get the following estimates for the function $g$.

\begin{lemma}\label{lemma:g}
	\begin{align}
		g(\a,Y) &\leq Y					& & \text{for all $\a$,}
		\label{gT}  \\
		g(\a,Y) &\ll \norm{\a}^{-1} 	& & \text{for all $\a$,} 
		\label{g1}	\\
		g(\a,Y)&= Y+O(1)				& & \text{if $\norm\a\leq Y^{-2}$}. 
		\label{g0}
	\end{align}
	where $\norm{\a}$ denotes the distance of $\a\in\TT$ from 0.
\end{lemma}
%

The estimates contained in \cref{lemma:g} imply that the integrand in \cref{def:R:total} is approximately equal to $e(-\a n)f_1f_2f_3f_4Y$ when $\a$ is close to 0, while it becomes ``small'' when $\a$ is bounded away from 0. 

\subsection{An approximation}
%

Under the assumption $\a\approx 0$ it is possible to approximate the Weyl sum $f(\a,X)$ with its ``mollification''
\begin{equation}
\nu(\a,X):= \sum_{\cc X^4< z\leq X^4} \tfrac 1 4 z^{-3/4}e(\a z),
\end{equation}
which is a weighted exponential sum that involves linear phases instead of biquadratic ones. 
%
From the book of Vaughan \cite{vaughan:book} we retrieve the following estimates.

\begin{lemma}
	\begin{align}
		&\nu(\a,X)\ll X 	& &  \text{for all $\a$,}
		\label{nuT} \\
		&\nu(\a,X)\ll X^{-3}\norm{\a}^{-1}		& &  \text{for all $\a$,}
		\label{nu1} \\	 
		&f(\a,X)\ll X 	& &  \text{for all $\a$,}
		\label{fT}	\\
		&f(\a,X) =\nu(\a,X) + O(1)		& &  \text{if $\norm \a \leq \cb X^{-3}$.}
		\label{f0} 
	\end{align}
\end{lemma}

\begin{proof}
	The estimates \eqref{nuT} and \eqref{nu1} are a restatement of \cite[Lemma 6.2]{vaughan:book}. 
	The estimate \eqref{fT} is trivial because $f(\a,X)$ is a sum of $O(X)$ exponentials. 
	Finally, \eqref{f0} follows from \cite[Lemma 6.1]{vaughan:book} with $q=1$.
\end{proof}

Then alongside $f_1,\dots, f_4$ we consider the mollified Weyl sums 
\begin{equation}\label{def:nu}
	\nu_j := \nu(\a,P_j).
\end{equation}
From \eqref{f0} we have that the approximation $f_j\approx \nu_j$ is admissible, up to an error of $O(1)$, on the interval $\AA j _0\subseteq \TT$ given by 
\begin{align}\label{BBj}
	\AA j_0 &=\{\a: \norm\a\leq \cb P_j^{-3}\}.
\end{align}
The complement of \eqref{BBj} in $\TT$ will be denoted by $\AA j_1$. 
In the range of small $\norm\a$ we also have $g\approx Y$: more precisely by \eqref{g0} and \eqref{hyp:gamma1} we have that $g-Y$ is bounded by an absolute constant on $\AA 1_0$ and $\AA 2_0$. 
Then, we consider the following integral

\begin{align}
U(n)&:= Y \int_\TT e(-n\a) \nu_1 \nu_2 \nu_3 \nu_4  d\a 
\label{def:U:total}
\end{align}

The integrand in \cref{def:U:total} is approximately equal to $e(-\a n)f_1f_2f_3f_4Y$ when $\a$ is close to 0, and it is small when $\a$ is bounded away from 0. 
Thus, by what we said at the end of the previous paragraph, we heuristically expect that $U(n)\sim R(n)$. 
We now show that $U(n)$ is in fact close to the expected value $\RR(n)$, up to an admissible error.

\begin{proposition}\label{UX}
	The following estimate holds uniformly for $n\in(\c N,N]$:
	\begin{align}
		U(n)-\RR(n) &\ll Y P_1^{-7}P_2^{5}P_3P_4 =  Y P^{-\tfrac{7131}{4096}}.		
	\end{align}
\end{proposition}

\begin{proof}
	By the definitions and by orthogonality, we have
	\begin{equation}\label{URR:dioph}
	U(n) = Y \sum_{\substack{
			\cc P_j^4 < z_j \leq P_j^4 \\
			z_1+z_2+z_3+z_4 = n
	}} 
	\tfrac 1 {256}(z_1 z_2 z_3 z_4) ^ {-3/4}.
	\end{equation}
	Since $P_j = o(P_1)$ for each $2\leq j\leq 4$, we have the inequality
	\begin{equation}\label{URR:P234}
		P_2^4+P_3^4+P_4^4  
		< \left(\c-\cc\right) P_1^4
	\end{equation}
	for all $P$ large enough. 
	Since moreover $\c P_1^4< n \leq P_1^4$, we have for every $n,z_2,z_3,z_4$ in the appropriate range that 
	\begin{equation}
	\cc P_1^4 < n-z_2-z_3-z_4 \leq P_1^4.
	\end{equation}
	In other words in \eqref{URR:dioph} we can safely express $z_1$ in terms of the other variables:
	\begin{equation}
	U(n) = \tfrac 1 {256}\sum_{\substack{
			z_2,z_3,z_4 \\
			\cc P_j^4 < z_j \leq P_j^4 
	}} 
	(z_2 z_3 z_4) ^ {-3/4}
	n^{-3/4} \left(1-\frac{z_2+z_3+z_4}{n}\right)^{-3/4}.
	\end{equation}
	We observe that $z_2+z_3+z_4 = O(P_2^4)$ and that 
	\begin{equation}
	\sum_{\cc P_j^4< z_j\leq P_j^4} \tfrac 1 4 z_j^{-3/4} 
	= \int_{\cc P_j^4}^{P_j^4} \tfrac 1 4 t^{-3/4} dt + O(P_j^{-3}),
	\end{equation}
	which is equal to $\c P_j (1+O(P_j^{-4})).$
	Since $P_2^{-4}\ll P_3^{-4}\ll P_4^{-4}\ll P_2^{4}P_1^{-4}$ we conclude that
	\begin{equation}
	U(n) = \tfrac 1 {32} Y P_2 P_3 P_4 n^{-3/4} (1+ O(P_2^{4}P_1^{-4})).
	\end{equation}
\end{proof}

\subsection{First application of the circle method}

For every $n\in\N$ and every measurable set $\BB\subseteq\TT$ (with respect to the natural Lebesgue-Haar measure) we define
\begin{align}
	R(n,\PP,\BB)&:=\int_\BB e(-n\a) f_1 f_2 f_3 f_4 g  d\a 
	\label{def:R}, \\
	U(n,\PP,\BB)&:= Y \int_\BB e(-n\a) \nu_1 \nu_2 \nu_3 \nu_4  d\a 
	\label{def:U}.
\end{align}

Since $Y^{-2}\geq \cb P_2^{-3}\geq \cb P_1^{-3}$ by \eqref{hyp:gamma1}, the approximations $g= Y+O(1)$ and $f_j=\nu_j+O(1)$ for $1\leq j\leq 4$ are valid when $\a\in\AA 1 _0$, where
\begin{equation}
\AA 1 _0 = [-\cb P_1^{-3},\cb P_1^{-3}].
\end{equation}

We define $\AA 1 _1$ to be its complement so that we have a partition $\TT=\AA 1 _0\sqcup \AA 1 _1$. Then we let $R_i(n):=R(n,\PP,\AA 1 _i)$ for $i\in\{1,0\}$ so that 
$R(n)=R_0(n)+R_1(n)$. 
In the remaining part of this section, we are going to prove that 
\begin{equation}\label{step1}
\abs{R(n)-\RR(n)} \leq \ER + \abs{R_1(n)}
\end{equation}
where $\ER$ is an error term satisfying the following estimate
\begin{equation}\label{step1:error}
\ER \lls  Y P^{-\tfrac{5083}{4096}} \approx  Y P^{-1.240967}.
\end{equation}

More precisely, we decompose $U(n)=U_0(n)+U_1(n)$ as we did for $R(n)$ via $U(n,\PP,\BB)$ and the partition $\TT=\AA 1 _0\sqcup \AA 1 _1$. 
Then by the triangular inequality  \eqref{step1} holds with
\begin{equation}
	\ER := \abs{R_0(n)-U_0(n)} + \abs{U_1(n)} + \abs{U(n)-\RR(n)}.
\end{equation}
The third absolute value was estimated in \cref{UX}; the other two terms are treated in the following propositions.


\begin{proposition}\label{U1U1}
	\begin{align}
		U_1(n) &\ll Y P_2^{-3}P_3P_4 =  Y P^{-\tfrac{5083}{4096}}.	
		\label{U1}
	\end{align}
\end{proposition}
\begin{proof}
By \eqref{nu1} applied to $\nu_1,\nu_2$ and \eqref{nuT} applied to $\nu_3,\nu_4$ we have
\begin{equation}
U_1(n) \ll YP_1^{-3}P_2^{-3}P_3P_4\int_{\AA 1 _1} \norm\a ^{-2} d\a
\end{equation}
and so \eqref{U1} follows from an elementary computation. 
\end{proof}

\begin{proposition}\label{R0R0}
	\begin{align}
		R_0(n)-U_0(n) 
		&\lls Y P_2^{-3}P_3P_4 
		=  Y P^{-\tfrac{5083}{4096}+\epsilon}.	
		\label{RU} 
	\end{align}
\end{proposition} 
\begin{proof}
Since $P_j^3\leq P_1^3$ for all $j$ and since $Y^2\leq 8 P_1^3$, we have by \eqref{f0} and \eqref{g0} 
\begin{equation}
R_0(n)-U_0(n) \ll
\int_{\AA 1 _0} 
	\left( 
		{\mu_1\mu_2\mu_3\mu_4} 
		+ Y({\mu_1\mu_2\mu_3} 
		+ {\mu_1\mu_2\mu_4} + {\mu_1\mu_3\mu_4} 
		+ {\mu_2\mu_3\mu_4}) 
	\right)
	d\a,
\end{equation}
where $\mu_j:=\max\{\abs{\nu_j},1\}$. 
We use \eqref{nuT}, i.e the trivial estimate $\mu_j\ll P_j$, on the factors with higher indices, to obtain
\begin{equation}
R_0(n)-U_0(n) \ll
P_2 P_3 P_4 \left(1 + \tfrac Y {P_2} + \tfrac Y {P_3} + \tfrac Y {P_4}\right) 
\int_{\AA 1 _0} {\mu_1}d\a 
+ YP_3P_4 \int_{\AA 1 _0} {\mu_2} d\a.
\end{equation}
Since $Y\geq P_4$ the factor that multiplies the first integral is $\asymp YP_2P_3$. Since $\mu_j\leq \abs{\nu_j}+1$ we can rewrite the last estimate as
\begin{equation}
R_0(n)-U_0(n) \ll Y P_2 P_3 \cdot P_1^{-3} +
YP_2 P_3 
\int_0^1 \abs{\nu_1}d\a 
+ YP_3P_4 \int_0^1 \abs{\nu_2} d\a.
\end{equation}

Then \cref{RU} follows from the following lemma, that we state separately for future reference, and the inequality $P_2P_1^{-3}=P^{-\tfrac{8960}{4096} } < P^{\tfrac{7787}{4096} } = P_4 P_2^{-3}$.
\end{proof}

\begin{lemma}\label{lemma:vi}
	\begin{align}
		\int_{\TT} \abs{\nu_j} d\a
		&\ll P_j^{-3} \log P_j .
		\label{vi} 
	\end{align}
\end{lemma}

\begin{proof}
	We estimate $\nu_j$ with 
	\begin{equation}\label{nu:estimating}
		\nu_j \ll 
		\begin{cases}
			P_j, 
			& \text{if $\norm \a\leq P_j^{-4}$, by \protect\eqref{nuT}},
			\\
			P_j^{-3}/\norm \a 
			& \text{otherwise, by \protect\eqref{nu1}}.
		 \end{cases}
	\end{equation}
	%
	Then the inequality follows from an elementary computation.
\end{proof}


\section{On the mean square deviation of $R(n)$}
\label{sec:S}

In this section we use Bessel's inequality to find an integral expression that bounds from above the average value of $\abs{R(n)-\RR(n)}^2$ for $n\in(N/2,N]$. 
We then perform a change of variables in the underlying arithmetic equation that makes the estimates on the absolute value of the integrand benefit from the restricted ranges $x_2,x_3,x_4\leq P_2=o(P_1)$. 
Finally we use again the circle method to estimate the error introduced by this change of variables. 

\subsection{Bessel's inequality}

From \eqref{step1} and the inequality $(A+B)^2\leq 2(A^2+B^2)$ we obtain that 
\begin{equation}\label{eq:pre:bessel}
\sum_{\c N< n\leq N} \abs{R(n)-\RR(n)}^2 \leq N \ER^2 + 2\sum_{\c N< n\leq N} \abs{R_1(n)}^2.
\end{equation}
In order to estimate the sum on the right, we use Bessel's inequality, as in \cite[eq.(12)]{daniel}, which in this case reveals that
\begin{equation}\label{eq:true:bessel}
	\sum_{\c N< n\leq N} \abs {R_1(n)}^2 \leq \int_{\AA 1 _1} \abs{f_1 f_2 f_3 f_4 g}^2 d\a.
\end{equation}
It is natural now to consider, for every measurable set $\BB\subseteq\TT$, the integral
\begin{equation}\label{def:S}
	S(\PP,\BB):=\int_\BB \abs{f_1 f_2 f_3 f_4 g}^2 d\a
\end{equation} 
and to let $S,S_0,S_1$ denote $S(\PP,\BB)$ respectively for $\BB = \TT,\AA 1 _0,\AA 1 _1$. 
With this notation, \cref{eq:pre:bessel} and \cref{eq:true:bessel} can be combined to give the inequality
\begin{equation}\label{eq:bessel}
\sum_{\c N< n\leq N} \abs{R-\RR}^2 \leq  N \ER^2 + 2S_1.
\end{equation}
We notice that this inequality has an underlying arithmetic meaning. 
In fact we have $S=S_0+S_1$ and we observe that $S$ counts the solutions to the equation 
\begin{equation}\label{eq:sysS}
x_1^4 + \dots + x_4 ^4+ y = {x'}_1^4 + \dots + {x'}_4 ^4+ y'
\end{equation}
subject to 
\begin{equation}\label{eq:subS}
0\less y,y'\leq Y,\quad \c P_i\less x_i,x'_i \leq P_i\quad (1\leq i \leq 4),
\end{equation} 
by orthogonality.

\subsection{A change of variables}\label{sec:T:def}

The equation \eqref{eq:sysS} can be rewritten in the following form
\begin{equation}\label{eq:sysT}
{(x_1+h)}^4 - {x_1}^4 = ({x_2}^4 -{x_2'}^4) + ({x_3}^4 -{x_3'}^4)+ ({x_4}^4 -{x_4'}^4) + (y - y'),
\end{equation}
where $h:=x_1'-x_1$. We now focus only on those solutions, subject to \eqref{eq:subS}, for which $h>0$. 
By orthogonality, their number $T$ is computed by the integral 
\begin{equation}\label{def:T}
	T=\int_\TT H_1 \abs{f_2 f_3 f_4 g}^2 d\a,
\end{equation}
where $H_1 = H(\a,P_1,\ch P_1^{-3}P_2^4)$ is an exponential sum associated to the difference polynomial $\Delta(x,h):= (x+h)^4 - x^4$:
\begin{equation}\label{def:H}
H(\a,X,Z) = \sum_{\substack{1\leq h\leq Z\\\c X < x \leq X-h}} e(\a [(x+h)^4 - x^4]).
\end{equation}

Indeed every such solution satisfies
\begin{equation}
h=x_1'-x_1 \leq ({x'_1}^4 - x_1^4) x_1^{-3} \leq 4P_2^4(\c P_1)^{-3}
\end{equation}
because of \cref{eq:sysT} and the inequalities $P_3^4,P_4^4,Y\leq P_2^4$. 
The number $S$ can be estimated by decomposing it naturally as $S=2T + (S-2T)$. The term $S-2T$ accounts for the solutions of \cref{eq:sysT} for which $h=0$, i.e. it corresponds to an equation in fewer variables, since $x_1$ can be eliminated. 
The term $2T$ instead is computed via the integral \eqref{def:T}. 
This is easier to estimate than the integral in \cref{def:S}, because its integrand is an exponential sum with fewer terms. 
Indeed $H_1$ only has $O(P_1^{-2}P_{2}^4)=O(P^{5/4})$ summands, which is noticeably less than the $O(P^2)$ terms of $\abs{f_1}^2$. In particular, we record that the trivial estimate
\begin{equation}
H_1(\a) \ll P_1^{-2} P_{2}^{4}
\label{H1T}
\end{equation}
holds uniformly for all $\a\in\TT$. 

\subsection{A mollified version of $\protect\abs{S-2T}$ near the origin}

Given the output \eqref{eq:bessel} of Bessel's inequality, we actually need to estimate the term $S_1$, which is a portion of the integral $S=S_0+S_1$ corresponding to the $\a$ that are bounded away from the origin. 
The idea is to decompose $T$ somewhat analogously as $T_0+T_1$ and then estimate $S_1$ as
\begin{equation}\label{eq:step2:s1}
S_1 \leq \abs{S_0-2T_0} + \abs{2T_1} + \abs{S-2T}.
\end{equation}
Since near the origin we have the estimates $g=Y+O(1)$ and $f_j=\nu_j+O(1)$, it is natural to compare the difference $S_0-2T_0$ with its mollified version $V-2W$, where
\begin{align}
	V &:= Y^2 \int_\TT \abs{f_1 \nu_2 \nu_3 \nu_4}^2 d\a, \\ 
	W &:= Y^2 \int_\TT H_1 \abs{\nu_2 \nu_3 \nu_4}^2 d\a.
\end{align}

Notice that we did not replace $f_1$ with its mollified version because we don't want to interfere with the change of variable that relates $\abs{f_1}^2$ to $H_1$. 
In the following proposition we estimate the difference $V-2W$ by looking at the underlying weighted diophantine equation. 
%
%
%

\begin{proposition}
	\begin{align}
		\label{VW}
		V-2W 
		&\ll 
		Y^2 P_1 P_2^{-2} P_3^{2} P_4^{2} =  Y^2. P^{\tfrac{7242}{4096}}
	\end{align}
\end{proposition}

\begin{proof}
	By orthogonality we have that
	\begin{equation}
	V = Y^2\sum_{n\in\Z} r(n) \rho(n)
	\end{equation}
	where $r(n)=r(n,P_1)$ is as in \eqref{def:r} and 
	\begin{equation}\label{def:rho}
	\rho(n) := \sum_{\substack{
			\cc P_j^4 < z_j,z'_j \leq P_j^4 \\
			z_2+z_3+z_4 - z'_2-z'_3-z'_4 = n	
	}}
	\tfrac 1 {4^6} (z_2z'_2z_3z'_3z_4z'_4)^{-3/4}.
	\end{equation}
	Similarly, we have
	\begin{equation}
	W = Y^2\sum_{n=1}^\infty r'(n) \rho(n)
	\end{equation}
	where $r'(n)$ is as in \eqref{def:r'}.
	We notice immediately that 
	\begin{align}\label{rho:0}
		\rho(n)&=0 && \text{for $\abs n > 3 P_2^4$}. 
	\end{align} 
	On the other hand we have
	\begin{align}\label{rr'}
		r(n)&=2r'(\abs n) && \text{for $0<\abs n \leq 4 P_2^4$}
	\end{align}
	because for $\c P_1< x,x'\leq P_1$ the inequality $\abs{{x'}^4 - x^4}\leq 4 P_2^4$ implies 
	\begin{equation}
	\abs{x'-x} \leq ({x'}^4 - x^4) \min\{x,x'\}^{-3} \leq \ch P_1^{-3}P_2^4. 
	\end{equation}
	In other words by \eqref{rho:0} and \eqref{rr'} we have
	\begin{equation}
	V-2W = Y^2 r(0) \rho(0). 
	\end{equation}
	Since $r(0) = \tfrac 1 2 P_1 + O(1)$ and 
	\begin{equation}
	\rho(0) \ll P_2^4 P_3^8 P_4 ^8 (P_2^8 P_3^8 P_4^8)^{-3/4}
	\end{equation}
	the proposition is proved. 
\end{proof}

\subsection{Some useful estimates}

Before we proceed to study the difference between $\abs{V-2W}$ and ``$\abs{S_0-2T_0}$'' (where $T_0$ has yet to be defined rigorously) we need to collect a few nontrivial estimates on integrals that involve $\abs{\nu_j}^2$, $\abs{f_j}^2$ and $H_1$. 
The first is similar to the one in \cref{lemma:vi}.
\begin{lemma}
	\begin{align}
		\int_{\TT} \abs{\nu_j}^2 d\a
		&\ll P_j^{-2}.
		\label{vii} 
	\end{align}
\end{lemma}
\begin{proof}
	We estimate $\nu_j$ as in \eqref{nu:estimating}, so that the inequality follows from an elementary computation.
\end{proof}

\begin{lemma}
	For every $A,B,X$ we have
	\begin{equation}\label{fi}
	\int_A^{A+B} \abs{f(\a,X)}^2 d\a \ll BX + X^{-2}\log X. 
	\end{equation}
\end{lemma} 
\begin{proof}
	The integral \eqref{fi} is estimated as in \cite[eq.(17)]{daniel} as follows.
	First, $\abs{f(\a,X)}^2 = \sum_{n\in\Z} r(n,X) e(\a n)$ where
	\begin{equation}\label{def:r}
	r(n,X) := \#\left\{(x,x')\left|\begin{matrix} {x'}^4 - x^4 =  n \\\  \c X\less x,x'\leq X \end{matrix}\right.   \right\}.
	\end{equation}
	Therefore
	\begin{equation}\label{eq:fi1}
	\int_{A}^{A+B} \abs{f(\a,X)}^2 d\a = \sum_{n\in \Z} r(n,X) \int_{A}^{A+B}e(\a n) d\a. 
	\end{equation}
	If $n\neq 0$ the change of variable $\b=\a n$ gives 
	\begin{equation}\label{eq:fi2}
	\int_{A}^{A+B}e(\a n) d\a
	= 
	\frac 1 n \int_{nA}^{nA+nB}e(\b) d\b 
	\leq \frac 2 {\abs n},  
	\end{equation}
	hence 
	\begin{equation}\label{eq:fi3}
	\int_{A}^{A+B} \abs{f(\a,X)}^2 d\a 
	= B r(0,X) 
	+ 
	O\left(\sum_{n\neq 0} \frac{r(n,X)}{\abs{n}}\right). 
	\end{equation}
	From the definition \eqref{def:r} we see that
	\begin{align}
		r(0,X)&\ll X, 
		&\\
		r(-n,X)&=r(n,X)
		& &\text {for all $n$,} \\
		r(n,X) &= 0 
		& &\text {for $0<\abs n\leq\tfrac 1 2 X^3$ or $\abs n >\tfrac {15}{16} X^4$.} 		
	\end{align} 
	Moreover we have that
	\begin{equation}
	\sum_{C<n\leq C+\tfrac 1 2 X^3} r(n,X)\leq X
	\end{equation}
	for every real $C$, because for every $x\in(X/2,X]$ there is at most one $x' \in(X/2,X]$ with $(C+x^4)< {x'}^4\leq (C+x^4)+\tfrac 1 2 X^3$. 
	As a consequence, we have \begin{equation}\label{rAB}
	\sum_{C<n\leq C+D} r(n,X)\leq 2DX^{-2} + O(X)	
	\end{equation} for all $C,D,X$. 
	Therefore 
	\begin{equation}\label{r/n}
	\sum_{n\neq 0} \frac{r(n,X)}{\abs{n}} 
	\leq 
	2 \sum_{k=-1}^{\lfloor\log_2 X\rfloor}
	\frac 1 {2^k X^3}  \sum_{2^k X^3<n\leq 2^{k+1}X^3} r(n,X)\ll X^{-2}\log X
	\end{equation}
	and \eqref{fi} follows.
\end{proof}

\begin{corollary}
	For all $1\leq j\leq 3$ we have
	\begin{align}
		\int_{\AA j_1} \abs{f_j}^2 \norm{\a}^{-2} d\a 
		&\ll P_j^4 \log P_j.
		\label{fB}
	\end{align}
\end{corollary}

\begin{proof}
	We divide the interval $\AA j_1$, defined under \eqref{BBj}, dyadically as follows
	\begin{equation}\label{dyadic:set}
	\AA j _1 \subseteq \bigcup_{k=-3}^{\lfloor 3\log_2 P_j\rfloor} 
	\{\a\in\TT:\ 2^k P_j^{-3}<\norm\a\leq 2^{k+1}P_j^{-3}\}
	\end{equation}
	into pairs of intervals of length at most $2^k P_j^{-3}$. Hence by \eqref{fi} we have
	\begin{equation}\label{dyadic:integral1}
	\int_{\AA j_1} \abs{f_j}^2 \norm{\a}^{-2} d\a 
	\ll
	\sum_{k=-3}^{\lfloor 3\log_2 P_j\rfloor} 
	P_j^{-2}(2^k + \log P_j) (2^{-2k}P_j^6)
	\end{equation}
	that gives \eqref{fB}.
\end{proof}

\begin{lemma}Let $\AA 2_1:=\{\a\in\R/\Z:\ \norm\a>\cb P_2^{-3} \}$ as per \eqref{BBj}, then
	\begin{equation}
	\label{HB}
	\int_{\AA 2_1} H_1 \norm \a ^{-4} d\a
	\ll P_1^{-2} P_2^{12}\log P_1 .
	\end{equation}
\end{lemma}

\begin{proof}
	We proceed as in the proof of \eqref{fi}. First, we notice that for every $A,B$
	\begin{equation}\label{Hi}
	\int_{A}^{A+B} H_1 d\a 
	\ll P_1^{-2} \log P_1.
	\end{equation}
	Indeed, $H_1(\a) = \sum_{n=1}^{\infty} r'(n) e(\a n)$ where
	\begin{equation}\label{def:r'}
	r'(n) := \#\left\{(h,x)
	\left|
	\begin{matrix} 
	(x+h)^4 - x^4 =  n 
	\\
	1\leq h\leq \ch
	P_1^{-3}P_2^4
	\\  
	\c P_1< x,x+h \leq P_1
	\end{matrix}
	\right.   \right\}.
	\end{equation}
	Therefore
	\begin{equation}\label{Hi:2} 
	\int_{A}^{A+B} H_1 d\a = \sum_{n=1}^{\infty} r'(n) \int_{A}^{A+B}e(\a n) d\a \ll \sum_{n=1}^{\infty} \frac{r'(n)}{n}
	\end{equation}
	as in \eqref{eq:fi1}-\eqref{eq:fi3}. It is clear from \eqref{def:r'} that 
	\begin{align}
		r'(n)& = 0 
		& & 
		\text{for $n\leq \tfrac 1 2 P_1^3$ or $n>\tfrac{15}{16} P_1^4$}
	\end{align}
	and arguing as for \eqref{rAB} we get
	\begin{equation}\label{r'AB}
	\sum_{A<n\leq A+B} r'(n)\leq 2BP_1^{-2} + O(P_1).	
	\end{equation}
	Then \eqref{Hi} follows from \eqref{Hi:2} and \eqref{r'AB} as in \eqref{r/n}. 
	Now we divide $\AA 2_1$ dyadically as in \eqref{dyadic:set} and we obtain
	\begin{equation}
	\int_{\AA 2_1} H_1 \norm{\a}^{-4} d\a 
	\ll
	\sum_{k=-3}^{\left\lfloor 3\log_2 P_2\right\rfloor} 
	P_1^{-2} \log P_1 \cdot 2^{-4k}P_2^{12}.
	\end{equation}
	The estimate \eqref{HB} follows.
\end{proof}

\subsection{From $S$ to $T$, through $V$ and $W$}\label{sec:STVW}

For every measurable set $\BB\subseteq\TT$, we recall the definition of teh integral $S(\PP,\BB)$ and we define $T(\PP,\BB)$ as follows:
\begin{align}
	S(\PP,\BB)&:=\int_\BB \abs{f_1f_2 f_3 f_4 g}^2 d\a, \\ 
	T(\PP,\BB)&:=\int_\BB H_1 \abs{f_2 f_3 f_4 g}^2 d\a.
\end{align}

We also recall that $S,S_0,S_1$ denote $S(\PP,\BB)$ respectively for $\BB=\TT,\AA 1_0,\AA 1_1$. 
We define $T=T_0+T_1$ analogously, but for the new partition $\TT=\AA 2_0\sqcup \AA 2_1$, where, as in \eqref{BBj}:
\begin{align}\label{BB2}
	\AA 2_0 &=\{\a: \norm\a\leq \cb P_2^{-3}\} & \AA 2_1 &=\{\a: \norm\a> \cb P_2^{-3}\} .
\end{align}

In view of \eqref{eq:step2:s1}, the goal of this section is to prove that 
\begin{equation}\label{eq:step2:s1:bis}
	\abs{S_0-2T_0}\ll Y^2 P^{ \tfrac {7242}{4096}} \approx Y^2 P^{1.768}.
\end{equation}
Notice that $\AA 1_0\subseteq \AA 2_0$ and that the approximations $g\approx Y$ and $f_j\approx \nu_j$ for $2\leq j\leq 4$ are valid on $\AA 2_0$, because $Y^{-2}\geq \tfrac 1 8  P_2^{-3}$ by \eqref{hyp:gamma1}. 
We introduce the following integrals

\begin{align}
	V(\PP,\BB)&:= Y^2 \int_\BB \abs{f_1 \nu_2 \nu_3 \nu_4}^2 d\a, \\ 
	W(\PP,\BB)&:= Y^2 \int_\BB H_1 \abs{\nu_2 \nu_3 \nu_4}^2 d\a,
\end{align}

then we define $V=V_0+V_1$ (resp. $W=W_0+W_1$)  using $V(\P,\BB)$ (resp. $W(\P,\BB)$) and the partition $\TT=\AA 1 _0\sqcup \AA 1 _1$  (resp. $\TT=\AA 2 _0\sqcup \AA 2 _1$). 
Then we have $\abs{S_0-2T_0}\leq \ES$, where 

\begin{equation}\label{ES}
	\ES := \abs{S_0-V_0} + \abs{V_1} + \abs{V-2W}+\abs{2W_1}+\abs{2W_0-2T_0}.
\end{equation}

We now dive into estimating the above five terms. 
%

\begin{proposition}\label{V1W1}
	\begin{align}
		V_1 &\lls Y^2 P_1^{4} P_2^{-6} P_3^{2} P_4^{2} &&=  Y^2 P^{\tfrac{6218}{4096}+\epsilon}	,
		\label{V1}
		\\
		W_1 &\lls Y^2 P_1^{-2} P_2^{6} P_3^{-6} P_4^2  &&=  Y^2 P^{-\tfrac{54}{4096}+\epsilon}	.
		\label{W1}
	\end{align}
\end{proposition}

\begin{proof}
	By \eqref{nu1} applied to $\nu_2$ and \eqref{nuT} applied to $\nu_3,\nu_4$ we have
	\begin{equation}
	V_1 \ll Y^2 P_2^{-6}P_3^2P_4^2\int_{\AA 1 _1} \abs{f_1}^2\norm\a ^{-2} d\a.
	\end{equation}
	which gives \eqref{V1} by \eqref{fB}. 
	By \eqref{nu1} applied to $\nu_2,\nu_3$ and \eqref{nuT} for $\nu_4$ we have
	\begin{equation}
	W_1 \ll Y^2 P_2^{-6}P_3^{-6}P_4^2\int_{\AA 2 _1} \abs{H_1}\norm\a ^{-4} d\a.
	\end{equation}
	The estimate \eqref{W1} follows by \eqref{HB}.
\end{proof}


\begin{proposition}\label{S0T0}
	\begin{align}
		S_0-V_0 
		&\ll Y^2 P_1^{-2} P_2^{2} P_3^{2} P_4 
		&&=  Y^2 P^{\tfrac{6069}{4096}}	,
		\label{SV} \\
		T_0-W_0 
		&\lls Y^2 P_1^{-2} P_2^{2} P_3^{2} P_4 
		&&=  Y^2 P^{\tfrac{6069}{4096}+\epsilon}	.
		\label{TW}
	\end{align}
\end{proposition}

\begin{proof}
	Analogously to the computation in \cref{R0R0}, by \eqref{f0} and \eqref{g0}  we have
	\begin{equation}
	S_0-V_0 \ll
	\int_{\AA 1 _0}
	\left(
		 Y\abs{f_1^2\mu_2^2\mu_3^2\mu_4^2}
		+ Y^2\abs{f_1^2}(\abs{\mu_2\mu_3^2\mu_4^2} 
		+ \abs{\mu_2^2\mu_3\mu_4^2} 
		+ \abs{\mu_2^2\mu_3^2\mu_4}) 
	\right)d\a,
	\end{equation}
	where $\mu_j:=\max\{\abs{\nu_j},1\}$. 
	We use the trivial estimate \eqref{nuT} for $\mu_2,\mu_3,\mu_4$ and we use that $f_1\ll\mu_1$ on $\AA 1_0$, by \eqref{f0}, to get
	\begin{equation}
	S_0-V_0 \ll
	Y^2P_2^2 P_3^2 P_4^2 \left(\tfrac 1 Y + \tfrac 1 {P_2} + \tfrac 1 {P_3} + \tfrac 1 {P_4}\right) 
	\int_{\AA 1 _0} \abs{\mu_1}^2 d\a.
	\end{equation}
	The integral to the right is $\ll P_1^{-2}$ by \eqref{vii} and the fact that $\int_{\AA 1_0} 1 d\a\ll {P_1^{-3}}$. 
	Since moreover $Y\geq P_4$, \eqref{SV} follows.  
	
	Similarly, since $P_j^3\leq P_2^3$ for all $j\geq 2$ and since $Y^2\leq 8 P_2^3$, we have by \eqref{f0} and \eqref{g0} 
	\begin{equation}
	T_0-W_0 \ll
	\int_{\AA 2 _0}
	\left(
	Y\abs{H_1\mu_2^2\mu_3^2\mu_4^2}
	+ Y^2\abs{H_1}(\abs{\mu_2\mu_3^2\mu_4^2} 
	+ \abs{\mu_2^2\mu_3\mu_4^2} 
	+ \abs{\mu_2^2\mu_3^2\mu_4}) 
	\right) 
	d\a.
	\end{equation}
	We apply \eqref{nuT} to $\mu_3,\mu_4$ and \eqref{H1T} to $H_1$ to get
	\begin{equation}
	T_0-W_0 \ll
	Y^2P_1^{-2}P_2^4 P_3^2 P_4^2 
	\left[
	\left(\tfrac 1 Y + \tfrac 1 {P_3} + \tfrac 1 {P_4}\right) 
	\int_{\AA 2_0} \abs{\mu_2}^2 d\a + \int_{\AA 2_0} \abs{\mu_2} d\a
	\right].
	\end{equation}
	The first integral is $\ll P_2^{-2}$ by \eqref{vii} while the second integral is $\lls P_2^{-3}$ by \eqref{vi}. 
	The expression inside the square brackets is therefore $\ll P_2^{-2} P_4^{-1}$, hence we get \eqref{TW}. 
\end{proof}

Finally, $V-2W$ was estimated in \cref{VW} and it turns out to be the main term in the right-hand side of \eqref{ES}. 
We conclude that 
\begin{equation}\label{ES:S0T0}
	\abs{S_0-2T_0} \leq \ES\ll Y^2 P^{\tfrac {7242}{4096} }.
\end{equation}


\section{Final estimates via the circle method}
\label{sec:T}

In this section we complete the proof of our main quantitative result, with a full application of the circle method and an induction on the number of variables in the underlying diophantine equation. 

\subsection{Induction on the number of variables}

At this point, we still need to estimate the terms $\abs{2T_1}$ and $\abs{S-2T}$ in \eqref{eq:step2:s1}. We already commented briefly on the fact that $S-2T$ counts the number of solutions to the equation \eqref{eq:sysS}, subject to \eqref{eq:subS}, together with $x_1'=x_1$. 
In particular if by $\S j$ we denote the number of solutions to the equation
 \begin{equation}\label{eq:sysS:2}
 x_j^4 + \dots + x_4 ^4+ y = {x'_j}^4 + \dots + {x'_j} ^4+ y'
 \end{equation}
 subject to 
 \begin{equation}\label{eq:subS:2}
 0\less y,y'\leq Y,\quad \c P_i\less x_i,x'_i \leq P_i\quad (j\leq i \leq 4),
 \end{equation} 
 we have $S-2T\asymp P_1 \S 2$. 
 Now, \cref{eq:sysS:2} has at least the ``diagonal'' solutions given by $y=y'$ and $x_i=x'_i$ for $j\leq i\leq 4$, hence 
 \begin{equation}\label{S:diagonal}
 	\S j \gg Y\prod_{i=j}^4 P_j.
 \end{equation}
 In particular, $S-2T\gg P_1 P_2 P_3 P_4 Y$ and we cannot hope for a better estimate of this term.
 In the remainder of the section we will prove, by backward induction on $j$, that in fact 
 \begin{equation}\label{S:diagonal:infact}
 \S j \lls Y\prod_{i=j}^4 P_j
 \end{equation}
 for $2\leq j\leq 4$ and then we will show that 
\begin{equation}\label{step3:TST}
	\abs{2T_1} + \abs{S-2T} \lls P_1P_2P_3P_4 Y =  YN^{1-\gamma_0+\epsilon/4},
\end{equation}
where $\gamma_0=4059/16384$. 
Since by 
\eqref{eq:bessel} we have 
\begin{equation}
		\sum_{\c N<n\leq N} \abs{R(n)-\RR(n)}^2 \leq N\ER^2 + 2\abs{S_0-2T_0} + 4\abs{4T_1} + 2\abs{S-2T},
\end{equation}
we finally get \cref{final:main} by using 
\eqref{step1:error}, \eqref{ES:S0T0} and \eqref{step3:TST}.  
The base step of induction is the following estimate of $\S 4$.

\begin{proposition}\label{S4}
	\begin{equation}
		\S 4 \ll P_4 Y.
	\end{equation}
\end{proposition}

\begin{proof}
	The number $\S 4$ counts the solutions to the equation
	\begin{equation}\label{eq:S4}
	x^4 + y = {x'}^4 + y'
	\end{equation}
	subject to $\c P_4< x,x'\leq P_4$ and $1\leq y,y'\leq Y$.
	For every such solution, say with $x\leq x'$, we have that
	\begin{equation}
	{x'}^4-x^4 \leq Y \leq \tfrac 1 2 P_4^3
	\end{equation}
	and so ${x'}^4-x^4 < {(x+1)}^4-x^4$. 
	This implies that \eqref{eq:S4} has only the diagonal solutions $x=x'$ and $y=y'$, therefore 
	\begin{equation}
	\S 4 = (\tfrac 1 2 P_4 + O(1)) (Y + O(1)).
	\end{equation}
\end{proof}

\subsection{Major arcs, central arc and minor arcs}

The equation \eqref{eq:sysS:2}, for $j\leq 3$ is transformed via the substitution $x_j'=x_j+h$, like we did in \cref{sec:T:def}. 
To the resulting equation
\begin{equation}\label{eq:sysT:2}
{(x_j+h)}^4 - x_j^4 = (x_{j+1}^4 -(x_{j+1}')^4) + \dots+ (x_4^4 -(x_4')^4) + (y - y')
\end{equation}
 additionally constrained by $h>0$, we attach the integrals
\begin{equation}\label{def:T:j}
\T j (\PP,\BB):=\int_\BB H_j\abs{g\prod_{i=j+1}^4 f_i}^2 d\a,
\end{equation}
where $\BB\subseteq\TT$ is a measurable set and where $$H_j:=H(\a,P_j,\ch P_j^{-3}P_{j+1}^4)$$ is given by \eqref{def:H}. 
The solutions to \eqref{eq:sysT:2} corresponding to $h=0$ are counted by 
\begin{equation}\label{scheme:ST}
	\S j - 2 \T j = (\tfrac 1 2 P_j + O(1)) \S {j+1}. 
\end{equation}
We are going to estimate the integrals \eqref{def:T:j} with the circle method. 

For every $1\leq j\leq 3$ and every pair of coprime integers $q,a$ with $q\geq 1$ we form 
\begin{equation}
\label{def:Mqa}
	\Mi j (q,a) := \{\a \in \TT: \norm{\a-a/q} \leq q^{-1} P_j P_{j+1}^{-4}\},
\end{equation}
and we define the $j$-th set of \emph{major arcs} by
\begin{equation}\label{eq:major}
\Mi j := \bigcup_{q=2}^{P_j} \bigcup_{\gcdaq}   \Mi j (q,a) .
\end{equation}
Notice that the intervals in the definition of $\Mi j$ are disjoint because for every two rational numbers $a/q$, $A/Q$ with denominators $q\leq Q\leq P_j$ we have
\begin{equation}
	\abs{\frac A Q - \frac a q}\geq \frac 1 {qP_j} \geq \frac 1 q P_j P_{j+1}^{-4} + \frac 1 Q P_j P_{j+1}^{-4}
\end{equation}
by \eqref{P:range}. 
Notice that in the definition of $\Mi j$ we excluded the major arc centered at zero. 
For $j\in\{2,3\}$ we denote the $j$-th \emph{central arc} by $\Ni j := \Mi j (1,0)$ and we define the $j$-th set of  \emph{minor arcs} $\mi j $ so that $\TT = \Ni j \sqcup \Mi j \sqcup \mi j$ is a partition. 
For $j=1$ we define the central arc by
\begin{equation} \label{def:N1}
	\Ni 1 := \{\a: \cb P_2^{-3}<\abs{\a} \leq P_1 P_{2}^{-4}\}
	= \Mi 1 (1,0) \cap \AA 2_1
\end{equation}
and consider the partition $\AA 2 _1 = \Ni 1 \sqcup \Mi 1 \sqcup \mi 1$. 
For every $1\leq j\leq 3$ we let $ \TN j, \TM j, \Tm j$ denote $\T j ( \PP,\BB)$ respectively for $\BB = \Ni j, \Mi j, \mi j$. 
Finally, we define $\T 1:= T_1$ and $\T j:=\T j(\PP,\TT)$ for $j\in\{2,3\}$, so that 
\begin{align}\label{T:circle:method}
	\T j 	& =  \TM j + \TN j + \Tm j 			& &(1\leq j \leq 3).
\end{align}

\subsection{Estimates for $H_j$ and the minor arc contribution}

It turns out that the minor arc component $\Tm j$ is the dominant term in $\T j$ for all $1\leq j\leq 3$. 
Nevertheless, we are going to estimate it crudely for each $1\leq j\leq 3$, as follows:
\begin{equation}\label{minor:brutal}
	\abs{\Tm j} \leq \left(\sup_{\a \in\mi j} \abs{H_j}\right) \int_{\mi j} \abs{f_{j+1}\cdots f_4 g}^2 d\a \leq \left(\sup_{\a \in\mi j} \abs{H_j}\right) \S {j+1}.
\end{equation}

	Thus we now need to bound from above the absolute value of the exponential sum $H_j$. Such estimate is proved as in \cite[Lemma ]{Vaughan:Weil} using the Weyl differencing method:

\begin{lemma}\label{H:lemma}
	Let $H(\a,X,Z)$ be as in \eqref{def:H} with $Z\leq X$ and  $\abs{\a-a/q}\leq q^{-2}$ for some integers $a,q$. 
	Then we have, for all $\epsilon>0$:
	\begin{equation}
	H(\a,X,Z) \ll_\epsilon X^{1+\epsilon} Z (X^{-1} + q^{-1} + qX^{-3}Z^{-1})^{1/4},
	\end{equation}
	where the implied constant depends only on $\epsilon$.
\end{lemma}

Since $H_j$ is a sum of terms with absolute value 1, it can be trivially estimated as 
\begin{equation}\label{HT}
	H_j(\a) \ll P_j^{-2} P_{j+1}^{4} = P_j^{5/4}
\end{equation}
for all $\a\in\TT$ and for each $1\leq j\leq 3$. 
From \cref{H:lemma} can deduce better pointwise estimates for $H_j$ in regions of interest to us. 
\begin{corollary}
	For all $1\leq j\leq 3$ and all $\a\in\Mi j (q,a)$ with coprime $q,a\leq P_j$ we have
	\begin{align}
		H_j(\a) &\lls  P_j^{-2} P_{j+1}^{4}\cdot q^{-1/4}.	
		\label{HM} 
	\end{align}
Moreover, for each $1\leq j\leq 3$ and all $\a\in\mi j$ we have 
	\begin{align} 
		H_j(\a) &\ll_\epsilon P_j^\epsilon P_j^{-2} P_{j+1}^{4}\cdot P_j^{-1/4} = P_j^{1+\epsilon}
		\label{Hm}. 
	\end{align}
\end{corollary}

\begin{proof}
	If $\a\in\Mi j (q,a)$ we apply \cref{H:lemma} and we get \eqref{HM} from $q\leq P_j$ and \eqref{P:range}.  
	Dirichlet's approximation theorem \cite[Lemma 2.1]{vaughan:book} says that for every $\a\in\R$ and every $Q\geq 1$ there are integers $a,q$ with $q\leq Q$ such that $\abs{\a-a/q}\leq 1/(qQ)$. 
	If $\a\in\mi j$ we apply Dirichlet's theorem with $Q=P_j^{-1}P_{j+1}^4$. The corresponding fraction $a/q$ satisfies $q>P_j$ by definition of $\mi j$ and so \cref{H:lemma} gives \eqref{Hm}. 
\end{proof}

\begin{remark}\label{rmk:HN1}
	By the same method, applying Dirichlet's theorem with $Q=P_2^3$, it is possible to prove that
	\begin{equation}
		H_1(\a) \lls P_1^{-2} P_{2}^{4}\cdot P_2^{-1/4}	
				\label{HN1} 
	\end{equation}
for $\a\in\Ni 1$. 
However, the trivial estimate $H_1(\a) \ll P_1^{-2} P_{2}^{4}$ will be sufficient for us in the treatment of the central arc $\Ni 1$.
\end{remark}

Focusing in particular on the minor arc estimate, for all $1\leq j\leq 3$ we get
\begin{equation}\label{scheme:minor}
	\Tm j \lls P_j \S {j+1}
\end{equation}
from \eqref{minor:brutal} and \eqref{Hm}. 
Combining \eqref{scheme:ST}, \eqref{T:circle:method} and \eqref{scheme:minor} we deduce that 
\begin{align}
	\S j &\lls P_j\S {j+1} + \TM j + \TN j & (2\leq j\leq 3),\\
	T_1 &\lls P_1\S {2} + \TM 1 + \TN 1. & 
\end{align}
This induction scheme, together with \eqref{scheme:ST} for $j=1$ and the base step \eqref{S4}, shows in particular that
%
%
\begin{equation}
	\abs{2T_1} + \abs{S-2T}\lls ( E_\mm + E_\MM + E_\NN ),
\end{equation} 
where
\begin{align}
	E_\MM &:= \abs{\TM 1} + P_1 \abs{\TM 2} + P_1 P_2 \abs{\TM 3},
	\label{def:EMM}\\
	E_\NN &:= \abs{\TN 1} + P_1 \abs{\TN 2} + P_1 P_2 \abs{\TN 3},  
	\label{def:ENN}\\
	E_\mm & := Y P_1 P_2 P_3 P_4 = Y P^{\tfrac{12325}{4096}}.
	\label{def:Emm}
\end{align}
Thus to prove the final estimate \eqref{step3:TST}, as well as the intermediate claims \eqref{S:diagonal:infact}, it is sufficient to prove that $\EM,\EN\ll \Em$. 

\subsection{Treatment of the central arc}

Here we estimate the error terms coming from the central arcs of $T^{(1)}$, $T^{(2)}$ and $T^{(3)}$. 
In order to prove that $\EN\ll E_\mm$ it is enough to show, since $Y\leq P^{\tfrac{4992}{4096}}$ by assumption, that $\EN\ll Y^2P^{\tfrac{7333}{16384}}$.

\begin{proposition}\label{TN}
	\begin{align}
		\TN 1 &\lls Y^2 P_1^{-2} P_2^{8} P_3^{-6} P_4^2 &=  Y^2 P^{\tfrac{6602}{4096}+\epsilon},	
		\label{TN1} \\
		P_1 \TN 2 &\lls Y^2 P_1 P_2^{-2} P_3^2 P_4^2 &=  Y^2 P^{\tfrac{7242}{4096}+\epsilon},
		\label{TN2} \\
		P_1 P_2 \TN 3 &\lls Y^2 P_1 P_2 P_3^{-1} P_4 &=  Y^2 P^{\tfrac{6917}{4096}+\epsilon}	.
		\label{TN3}
	\end{align}
\end{proposition}

\begin{proof}
	We have
	\begin{equation}
	\TN 1 \leq \left(\sup_{\a \in\Ni 1}\abs{H_1 f_4^2 g^2}\right) \int_{\Ni 1} \abs{f_2 f_3}^2 d\a.
	\end{equation}
	We also have $\Ni 1 \subseteq \AA 2_1\cap  \AA 3_0$ (see \eqref{BBj} and \eqref{def:N1}) since the inequalities
	\begin{equation}
		\cb P_2^{-3} < \norm \a \leq P_1P_2^{-4} \leq \cb P_3^{-3}
	\end{equation}
	hold for every $\a\in\Ni 1$. In particular $f_3$ is well approximated by $\nu_3$ on $\Ni 1$ and so $f_3(\a)\ll P_3^{-3}\norm\a ^{-1}$ by \eqref{f0} and \eqref{nu1}. 
	Therefore
	\begin{equation}
	\int_{\Ni 1} \abs{f_2 f_3}^2 d\a \ll P_3^{-6} \int_{\AA 2_1} \abs{f_2}^2 \norm{\a}^{-2} d\a
	\end{equation}
	which is $\lls P_2^{4}P_3^{-6}$ by \eqref{fB}.  
	Hence \eqref{TN1} follows using the trivial estimates $H_1\ll P_1^{-2} P_2^4$, $g\ll Y$ and $f_4\ll P_4$. \footnote{We could have saved $P_2^{-1/4}$ by using the more precise estimate \protect\eqref{HN1}, but this is not much actually.}  
	We deal with $\TN 2$ similarly:
	\begin{equation}
	\TN 2 \leq \left(\sup_{\a \in\Ni 2}\abs{H_2 g^2}\right) \int_{\Mi 2 (1,0)} \abs{f_3 f_4}^2 d\a.
	\end{equation} 
	We estimate $H_2$ and $g$ trivially as above. 
	To estimate the integral instead, we observe that $\Mi 2 (1,0) \subseteq \AA 3 _0 \sqcup (\AA 4_0 \setminus \AA 3_0)$. 
	On the interval $\AA 3_0$ we estimate $f_4$ trivially, 
	while on $\AA 4_0 \setminus \AA 3_0$ we proceed as in the previous case, so
	\begin{equation}\label{eq:TN2}
		\TN 2 \ll Y^2 P_2^{-2} P_3^4
		\left( 
		P_4^2\int_{\AA 3_0} \abs{f_3}^2d\a 
		+ 
		P_4^{-6}\int_{\AA 3_1} \abs{f_3}^2 \norm{\a}^{-2} d\a  
		 \right).
	\end{equation}
	Since on $\AA 3_0$ the approximation $f_3=\nu_3+O(1)$ holds, we have $\abs{f_3}^2 = \abs{\nu_3}^2 + O(P_3)$ and so the first integral in \eqref{eq:TN2} is estimated as
	\begin{equation}
		\int_{\AA 3 0} \abs{f_3}^2d\a 
		\ll \int_{\TT} \abs{\nu_3}^2 d\a + P_3\cdot \int_{\AA 3 0} 1d\a,
	\end{equation}
	which is $\ll P_3^{-2}$ by \eqref{vii}. 
	On the other hand the second integral of \eqref{eq:TN2} is $\ll P_3^4 \log P_3$ by and \eqref{fB}, so \eqref{TN2} follows. 
	Finally \eqref{TN3} follows simply from 
	\begin{equation}
	\TN 3 \leq \left(\sup_{\a \in\Ni 3}\abs{H_3 g^2}\right) \int_{\Mi 3 (1,0)} \abs{f_4}^2 d\a,
	\end{equation} 
	estimating $H_3$ and $g$ trivially and using \eqref{fi} with $B=2 P_3P_4^{-4}$ to estimate the integral. 
\end{proof}	

\subsection{Treatment of the major arcs}

Here we estimate the error terms coming from the major arcs in $\Mi j$ (which exclude the central one). 
Since the Weyl sum $g$ is small away from $0$, we are able to estimate it nontrivially on $\Mi j$. 
For example we have the following proposition, that is obtained, \emph{mutatis mutandis}, from  \cite[Lemma 2]{daniel}.

\begin{proposition}
	For all $1\leq j\leq 3$ we have, uniformly on $q>1$:
	\begin{equation} \label{gM}
		\sum_{\gcdaq}
		 \left(\sup_{\a\in\Mi j (q,a)} \abs{g(\a,Y)}^2\right)
		\ll q Y. 
	\end{equation}
\end{proposition}


This allows us to save one power of $Y$ in the estimate for $\EM$. 
We will need also some estimates for the Weyl sums $f_j$. 
For this purpose the following result, taken from the book of Vaughan \cite{vaughan:book}, is very useful.


\begin{lemma}\label{label:vaughan:major}
	For every coprime $q,a$ and every $\epsilon>0$ we have
	\begin{align}
		f(a/q + \b,X) &\ll q^{-1/4} \nu(\b,X) + q^{1/2+\epsilon}  (1+X^4 \norm\b)^{1/2}
		& &  \text{for all $\b$,}
		\label{fV+} \\
		f(a/q + \b,X) &\ll q^{-1/4} \nu(\b,X) + q^{1/2+\epsilon} 
		& &  \text{if $\norm \b < \frac 1 {8 q X^3}$.}
		\label{fV-}
	\end{align}
\end{lemma}

\begin{proof}
	The estimates \eqref{fV+} and \eqref{fV-} follow from \cite[Thm 4.1 and Thm 4.2]{vaughan:book}. 
\end{proof}

In our case \cref{label:vaughan:major} is used to estimate the $f_j$ in absolute value and in mean square over the major arcs, as in the following two corollaries.

\begin{corollary}
	For all $1\leq j\leq 3$, all coprime $q,a\leq P_j$ and all $\epsilon >0$ we have
	\begin{align}
		\int_{\Mi j (q,a)} \abs{f_{j+1}}^2 d\a
		&\lls q^{-1} P_j^{1/2} P_{j+1}^{-2}
		\label{fM} 
	\end{align}
\end{corollary}

\begin{proof}
	By \eqref{fV+} and \eqref{def:Mqa} we have
	\begin{equation}
		\int_{\Mi j (q,a)} \abs{f_{j+1}}^2 d\a 
		\ll 
		q^{-1/2} \int_\TT \abs{\nu_{j+1}}^2 d\a 
		+ 
		q^{\epsilon}P_j\int_{\Mi j (q,a)} 1\,  d\a 
	\end{equation}
	and so \eqref{fM} follows by \eqref{vii},  \eqref{P:range} and $q\leq P_j$.
	%
\end{proof}

\begin{corollary}
	For all $1\leq i,j\leq 4$ with $j\geq i+2$ and all coprime $q,a\leq P_i$ we have
	\begin{equation}
		\label{fM-}
		\sup_{\a\in\Mi i (q,a)} \abs{f_j(\a)} \lls q^{-1/4} P_i^{3/4}.
	\end{equation}
\end{corollary}

\begin{proof}
	For $\a\in\Mi i (q,a)$ we may estimate $\abs{f_j(\a)}$ with \eqref{fV-} because the inequality 
	\begin{equation}\label{hpM}
	\frac 1 q P_iP_{i+1}^{-4}< \frac 1 {8q} P_{j}^{-3}
	\end{equation}
	holds for $P$ large enough. 
	Then \eqref{fM-} follows from the trivial estimate $\nu(\beta,P_j)\ll P_j$ and the inequality $P_j\leq P_i^{3/4}$. 
\end{proof}

We are now ready for the last computations. We recall that in order to have $\EM\ll E_\mm$ we need to show that  $\EM\ll Y P^{\tfrac{12325}{4096} }$. 

\begin{proposition}\label{TM}
	\begin{align}
	\TM 1 &\lls Y P_1^{5/4} P_2^{2} &= 
	Y P^{\tfrac{11776}{4096}+\epsilon},
	\label{TM1}
	\\
	P_1 \TM 2 &\lls Y P_1 P_2^{1/4} P_3^{2} &=  
	Y P^{\tfrac{10328}{4096}+\epsilon},
	\label{TM2}
	\\
	P_1 P_2 \TM 3 &\lls Y P_1 P_2 P_3^{-3/4} P_4^2 &=  Y P^{\tfrac{9790}{4096}+\epsilon}. 
	\label{TM3}	
	\end{align}
\end{proposition}

\begin{proof}
	From the definitions we have 
	\begin{equation}
		\TM 1 \leq \sum_{\substack{2\leq q\leq P_1\\a\in(\Z/q\Z)^\ast}} 
		\sup_{\a\in\Mi 1(q,a)} \abs{g}^2 
		\cdot \left( \sup_{\a\in\Mi 1(q,a)} \abs{H_1 f_3^2 f_4^2}\right)
		\int_{\Mi 1 (q,a)} \abs{f_2}^2 d\a.
	\end{equation}
	We apply \eqref{gM} to $g$,  \eqref{HM} to $H_1$ and  \eqref{fM-} to estimate $f_3,f_4$. Together with \eqref{fM} we get
	\begin{equation}
		\TM 1 \lls 
		\sum_{q=2}^{\lfloor P_1\rfloor}
		q Y 
		\cdot
		q^{-1/4} P_1^{-2}P_2^4
		\cdot
		q^{-1/2} P_1^{3/2}
		\cdot
		q^{-1/2} P_1^{3/2}
		\cdot 
		q^{-1} P_1^{1/2} P_2^{-2}
	\end{equation} 
	which gives \eqref{TM1}. 
	Similarly, to estimate $\TM 2$ we apply \eqref{gM} to $g$, \eqref{HM} to $H_2$, \eqref{fM-} 
	 to $f_4$ and \eqref{fM} to $f_3$:
	\begin{equation}
		\TM 2 \lls 
		\sum_{q=2}^{\lfloor P_2\rfloor}
		q Y 
		\cdot
		q^{-1/4} P_2^{-2}P_3^4
		\cdot
		q^{-1/2} P_2^{3/2}
		\cdot 
		q^{-1} P_2^{1/2} P_3^{-2}
	\end{equation}
	that gives \eqref{TM2}. 
	Finally, again by \eqref{gM}, \eqref{HM} and \eqref{fM} 
	we have
	\begin{equation}
		\TM 3 \lls 
		\sum_{q=2}^{\lfloor P_3\rfloor}
		q Y 
		\cdot
		q^{-1/4} P_3^{-2}P_4^4
		\cdot 
		q^{-1} P_3^{1/2} P_4^{-2}
	\end{equation}
	that gives \eqref{TM3}. 
\end{proof}

%
%
%
%
%
%
%
%
%

\bibliographystyle{plain}
\bibliography{biblio_danielgaps4_v2}

\begin{thebibliography}{10}

\bibitem{Vino:proof}
J.~Bourgain, C.~Demeter, and L.~Guth.
\newblock {Proof of the main conjecture in Vinogradov's mean value theorem for
  degrees higher than three}.
\newblock {\em Annals of Mathematics}, 184:633--682, 2016.

\bibitem{BWshort2}
J.~Br{\"u}dern and T.D. Wooley.
\newblock {Additive representation in short intervals, II: sums of two like
  powers}.
\newblock {\em Math. Z.}, 286:179--196, 2017.

\bibitem{daniel}
S.~Daniel.
\newblock {On gaps between numbers that are sums of three cubes}.
\newblock {\em Mathematika}, 44(1):1--13, 1997.

\bibitem{davenport:biquadrates}
H.~Davenport.
\newblock {On Waring's problem for fourth powers}.
\newblock {\em Ann. Math.}, 40:731--747, 1939.

\bibitem{deshouillers:random}
J.M. Deshouillers, F.~Hennecart, and B.~Landreau.
\newblock {Sums of powers: an arithmetic refinement to the probabilistic model
  of Erd{\H{o}}s and R{\'{e}}nyi}.
\newblock {\em Acta Arithmetica}, 85(1):13--33, 1998.

\bibitem{deshouillers:cubes}
J.M. Deshouillers, F.~Hennecart, and B.~Landreau.
\newblock On the density of sums of three cubes.
\newblock In {\em Algorithmic number theory}, volume 4076 of {\em Lecture Notes
  in Comput. Sci.}, pages 141--155. Springer, Berlin, 2006.

\bibitem{erdos:random}
P.~Erd{\H{o}}s and A.~R{\'{e}}nyi.
\newblock {Additive properties of random sequences of positive integers}.
\newblock {\em Acta Arithmetica}, 6(1):83--110, 1960.

\bibitem{ghidelli:gaps}
L.~Ghidelli.
\newblock Arbitrarily long gaps between the values of positive-definite cubic
  and biquadratic diagonal forms.
\newblock {\em Preprint}, 2019.

\bibitem{ghidelli:theta}
L.~Ghidelli.
\newblock Arithmethic properties of values of cubic and biquadratic theta
  functions.
\newblock {\em Preprint}, 2019.

\bibitem{HL:diminishing}
G.H. Hardy and J.E. Littlewood.
\newblock {Some problems of “Partitio Numerorum” (VI): Further researches
  in Waring’s problem}.
\newblock {\em Math. Z.}, 23:1--37, 1925.

\bibitem{heathbrown:cubes}
D.R. Heath-Brown.
\newblock The circle method and diagonal cubic forms.
\newblock {\em {Philosophical Transactions of the Royal Society of London.
  Series A: Mathematical, Physical and Engineering Sciences}},
  356(1738):673--699, 1998.

\bibitem{hooley:cubes:L1}
C.~Hooley.
\newblock {On Waring's problem}.
\newblock {\em Acta Mathematica}, 157:49--97, 1966.

\bibitem{hooley:sometopics}
C.~Hooley.
\newblock {On some topics connected with Waring's problem}.
\newblock {\em J. reine angew. Math}, 369:110--153, 1986.

\bibitem{hooley:cubes:L2}
C.~Hooley.
\newblock {On Hypothesis $K^\ast$ in Waring's problem}.
\newblock In {\em {Sieve methods, exponential sums, and their applications in
  number theory, Cardiff, 1995. London Math. Soc. Lecture Series}}, volume 237,
  pages 175--185. Cambridge University Press, 1997.

\bibitem{Vino:expo}
L.B. Pierce.
\newblock {The Vinogradov mean value theorem after Wooley, and Bourgain,
  Demeter and Guth)}.
\newblock {\em S{\'e}minaire Bourbaki, 69i{\`e}me ann{\'e}e}, pages 1134--1179,
  Juin 2017.

\bibitem{Vaughan:Weil}
R.C. Vaughan.
\newblock {On Waring’s problem for smaller exponents}.
\newblock {\em Proc. Lond. Math. Soc.}, 52(3):445--463, 1986.

\bibitem{vaughan:book}
R.C. Vaughan.
\newblock {\em {The Hardy-Littlewood method}}.
\newblock Number~2 in Cambridge tracts in mathematics. Cambridge University
  Press, 2 edition, 1997.

\bibitem{waring:survey}
R.C. Vaughan and T.D. Wooley.
\newblock {Waring's problem: a survey}.
\newblock {\em Number theory for the millennium 3}, pages 301--340, 2002.

\bibitem{wooley:cubes}
T.D. Wooley.
\newblock {Sums of three cubes}.
\newblock {\em Mathematica}, 47:53--61, 2000.

\bibitem{Vino:Woo}
T.D. Wooley.
\newblock {Nested efficient congruencing and relatives of Vinogradov's mean
  value theorem}.
\newblock {\em Proceedings of the London Mathematical Society},
  118(4):942--1016, 2019.

\end{thebibliography}

\end{document}